\documentclass[12pt]{article}
\usepackage{amsmath, amssymb, amsthm}
\newtheorem{theorem}{Theorem}
\newtheorem{definition}[theorem]{Definition}
\newtheorem{example}[theorem]{Example}
\newtheorem{problem}{Problem}

\title{Random Relation Algebras}
\author{Jeremy F.~Alm}
\date{February 2018}

\begin{document}

\maketitle

What does a ``typical" finite relation algebra look like? In graph theory, one has the ``random graph" $G_{n,p}$, which is actually a probability space of graphs \cite{MR0120167}. (If one sets $p=\frac{1}{2}$, $G_{n,p}$ corresponds to the uniform distribution on the set of all labelled graphs on $n$ vertices.) Then a graph property $P$ (like being connected) is said to hold in ``most" graphs if the probability that $P$ holds in $G_{n,p}$ goes to one as $n\rightarrow \infty$.

In this paper, we develop a random model for finite symmetric integral relation algebras, and prove some preliminary results.

\begin{definition}
Let $R(n,p)$ denote the probability space whose events are the finite symmetric integral not-necessarily-associative relation algebras with $n$ diversity atoms. For each diversity cycle $abc$, make it mandatory with probability $p$ (and forbidden otherwise), with these choices independent of one another.

\end{definition}

\begin{example} Let $n=3$, and $p=\frac{1}{2}$. Given diversity atoms $a,b,c$, the possible diversity cycles are $aaa,bbb,ccc,abb,baa,acc,caa,bcc,cbb,abc$. The random selection of all cycles except $bbb$ and $cbb$ gives relation algebra $59_{65}$, while the selection of only $abb,acc$, and $bcc$ gives $1_{65}$. Clearly, \emph{some} selections will fail to give a relation algebra.
\end{example}
\begin{theorem}
For any fixed $0<p\leq 1$, the probability that $R(n,p)$ is a relation algebra goes to one as $n\rightarrow \infty$.
\end{theorem}

\begin{proof}
We must show that $R(n,p)$ is associative, for which it suffices to show the following: for all mandatory $abc$ and $xyc$, there is a $z$ such that $axz$ and $byz$ are mandatory. There are $n+2{n\choose 2} +{n\choose 3}$ diversity cycles, which is asymptotically $\frac{n^3}{6}$. There are thus ${\frac{n^3}{6}\choose 2}$ possible pairs of cycles, which is asymptotically $\frac{n^6}{72}$. (This is over-counting, since some of those pairs won't ``match up" with a common diversity atom, but it won't matter.) For any given pair $abc,xyc$, the probability that, for a particular atom $z$, $axz$ and $byz$ are not both mandatory is $1-p^2$. The probability that \emph{no} such $z$ works is then $\Pi_z(1-p^2)$. Hence the overall probability of failure of associativity is bounded above by 

\[
    \sum_{\substack{abc \\ xyc}} \prod_z(1-p^2)=\sum_{\substack{abc \\ xyc}}(1-p^2)^n,
\]
which is asymptotically $\frac{n^6}{72}(1-p^2)^n$, which goes to zero for fixed $p$.

\end{proof} 

Now we turn to the question of representability. We use the fact that having a flexible atom is sufficient for representability over a countable set.

\begin{theorem}\label{thm:3}
Let $p\geq n^{\frac{-1}{{n+1\choose 2}}}$. Then the expected number of flexible atoms is $R(n,p)$ is at least one.
\end{theorem}

\begin{proof}
Given an atom $z$, the probability that it is flexible is $p^{n+1\choose 2}$, since all of the ${n+1\choose 2}$ cycles involving $z$ must be mandatory. Then by linearity of expectation we have
\[
    \mathbb{E}[\text{number of flexible atoms}]=\sum_z p^{{n+1\choose 2}}=np^{{n+1\choose 2}}.
\]

Set $p\geq n^{\frac{-1}{{n+1\choose 2}}}$. Then $np^{{n+1\choose 2}}\geq n\big({n^{\frac{-1}{{n+1\choose 2}}}\big)^{{n+1\choose 2}}}=1$.

\end{proof}

Theorem 3 has two rather glaring shortcomings. First, it doesn't show that the probability of representability goes to one as $n\rightarrow \infty$, as one usually wants. Second, using the presence of a flexible atom as a sufficient condition for representability is overkill. It seems like it ought to be possible to strengthen Theorem \ref{thm:3} to prove that almost all finite symmetric integral relation algebras are representable, and a more general definition of $R(n,p)$ might allow a positive solution to problem 20 from \cite{HH}:  If $RA(n)$ (respectively, $RRA(n)$) is the number of isomorphism types of relation algebras (respectively, representable relation algebras) with no more than $n$ elements, is it the case that
\[
    \lim_{n\to\infty}\frac{RRA(n)}{RA(n)}=1 ?
\]

However, what is really desired (by this author, at least) is a notion of a \emph{quasirandom} relation algebra.  There are many graph properties, all asymptotically equivalent, that hold almost surely in $G_{n,1/2}$ and therefore can be taken as a definition of a quasirandom graph.  One such example is the property of having all but $o(n)$ vertices of degree $(1+o(1))\frac{n}{2}$.  Such properties serve as proxies for ``randomness''.

In a similar fashion, quasirandom subsets of $\mathbb{Z}/n\mathbb{Z}$ were defined in \cite{MR1178385}.  Again, a number of properties were proved to be asymptotically equivalent.  One such property is that of the characteristic function of the subset having small (as in $o(n)$) nontrivial Fourier coefficients.

What would be a quasirandom relation algebra?  Restricting attention once again to symmetric integral relation algebras, here is one possibility.  For each atom $a$, form a graph $G_a$ with vertices labeled with the other diversity atoms, with an edge between $b$ and $c$ if $abc$ is mandatory (or a loop on $b$ if $abb$ is mandatory).  Then call the algebra quasirandom if all but $o(n)$ of the graphs $G_a$ are quasirandom.

Is this a good definition?  Probably not.  (It completely ignores 1-cycles, for example.  Does that matter? The fraction of diversity cycles that are 1-cycles is asymptotically zero.)  I offer it merely as an example of the sort of thing one might propose.  My purpose is to start a conversation that might lead to a significant interaction between the field of relation algebra and the subfield of combinatorics that is concerned with quasirandom structures.  This paper is a first step.

Here are a few problems to consider.

\begin{problem}
Is there a function $p(n)$ such that $R(n,p(n))$ is asymptotically the uniform distribution on symmetric integral relation algebras of order $2^{n+1}$?
\end{problem}

\begin{problem}
Improve the bound on $p$ in Theorem \ref{thm:3}.
\end{problem}

\begin{problem}
Formulate several notions of quasirandomness for relation algebras, and show that they are equivalent, as in \cite{MR1178385, MR1054011}. Maddux's work on algebras with no mandatory 3-cycles \cite{MR2281590} suggests that the difficult part of representability lies in the 3-cycles.  Results on quasirandom 3-uniform hypergraphs might be relevant.
\end{problem}

\begin{problem}
First-order graph properties obey a 0-1 law in the standard uniform random graph model, i.e., every property holds with asymptotic probability 1 or asymptotic probability 0 in $G_{n,1/2}$. Does the same hold for $R(n,p)$?
\end{problem}

\bibliographystyle{plain}
\bibliography{refs}
\end{document}